\documentclass[preprint,12pt]{elsarticle}

\usepackage[utf8]{inputenc}
\usepackage[T1]{fontenc}
\usepackage{amssymb,latexsym}
\usepackage{amsmath}
\usepackage{amsthm}
\usepackage{amssymb}
\usepackage{color}
\usepackage[shortlabels]{enumitem}
\usepackage[colorlinks=true,citecolor={blue},draft=false, urlcolor={black}]{hyperref}
\usepackage{cleveref}
\usepackage{tkz-graph}
\usetikzlibrary{arrows}
\usetikzlibrary{calc}

\usepackage{ifdraft}
\usepackage{listings}%
\usepackage{esvect} 

\usepackage{tikz}

\def\N{\mathbb N}
\def\A{\mathcal A}
\def\B{\mathcal B}

\def\uu{\mathbf{u}}

\def\yy{\mathbf{y}}
\def\zz{\mathbf{z}}
\def\xx{\mathbf{x}}

\newtheorem{thm}{Theorem}
\newtheorem{theorem}[thm]{Theorem}

\newtheorem{lemma}[thm]{Lemma}

\newtheorem{proposition}[thm]{Proposition}

\newtheorem{defi}[thm]{Definition}

\crefname{thm}{theorem}{theorems}
\crefname{theorem}{theorem}{theorems}
\crefname{coro}{corollary}{corollaries}
\crefname{example}{example}{examples}
\crefname{lemma}{lemma}{lemmas}
\crefname{claim}{claim}{claims}
\crefname{obs}{observation}{observations}
\crefname{proposition}{proposition}{propositions}
\crefname{prop}{proposition}{propositions}
\crefname{defi}{definition}{definitions}
\crefname{rem}{remark}{remarks}

\newtheorem{remark}[thm]{Remark}
\newtheorem{example}[thm]{Example}

\begin{document}

\begin{frontmatter}

\title{ A note on symmetries of rich sequences with minimum critical exponent}

\author[fn]{Ľubom\'ira Dvo\v r\'akov\'a\corref{ss}}\ead{lubomira.dvorakova@fjfi.cvut.cz}

\author[fn]{ \ Edita  Pelantov\'a}

\affiliation[fn]{organization={Department of Mathematics, FNSPE, Czech Technical University in Prague},
            addressline={Trojanova 13}, 
            city={Prague},
            postcode={12000}, 
            country={Czech Republic}}

\cortext[ss]{corresponding author}


\begin{abstract} 
Using three examples of sequences over a finite alphabet, we want to draw attention to the fact that these sequences having the minimum critical exponent in a given class of sequences show a large degree of symmetry, i.e., they are $G$-rich with respect to a group G generated by more than one antimorphism.
The notion of $G$-richness generalizes the notion of richness in palindromes which is based on one antimorphism, namely the reversal mapping. The three examples are: \\ 1)~the Thue-Morse sequence which has the minimum critical exponent among all binary sequences; \\ 2)~the sequence  which has the minimum critical exponent among all binary rich sequences;\\  3)~the sequence  which has the minimum critical exponent among all ternary rich sequences.

\end{abstract}



\end{frontmatter}




\section{Introduction}\label{sec:Introduction}
Droubay, Justin and Pirilo~\citep{DrJuPi2001} showed that a~word $w$ of length $n$ contains at most $n+1$ distinct palindromes.   If this bound is reached, we say that $w$ is {\em rich in palindromes}, or simply {\em rich}. They also proved that each factor of a rich word is again rich, i.e., the language of rich words is factorial. 
A sequence (infinite word) ${\bf w} = w_0w_1w_2 \cdots$ is called {\em rich in palindromes} if each factor of $\bf w$ is rich. The most popular sequence rich in palindromes is the Fibonacci sequence. 
A {\em palindrome}, which is the central notion in the definition of palindromic richness, may be formally defined as a word $w=w_0w_1\cdots w_{n-1}$ for which its reversal $R(w)=w_{n-1}w_{n-2}\cdots w_0$ is equal to $w$.  If we choose a different involutive antimorphism $\Theta$, we may analogously define $\Theta$-richness.  Guo, Shallit and Shur~\citep{Guo2016} studied binary sequences rich in palindromes in the classical sense and also $E$-rich sequences, where $E$ is an antimorphism on the binary alphabet exchanging letters. They showed that the notion of $E$-richness is not interesting as every $E$-rich word is a~factor of the periodic sequence $(01)^\omega$.  

Another generalization of richness was introduced in~\citep{PStarosta2013},  where instead of one antimorphism, a group $G$ generated by involutive antimorphisms is considered. A word $w$ is called $G$-{\em palindrome} if $w = \Theta(w)$ for some involutive antimorphism $\Theta\in G$. The number of distinct $G$-palindromes occurring in a word $w$ is bounded from above by a term depending on $G$ and on the length of $w$. If this upper bound is attained, $w$ is called $G$-{\em rich}. Analogously, if every factor of a sequence ${\bf w}$ is $G$-rich, the sequence is said to be $G$-{\em rich}. $G$-richness is well-understood on the binary alphabet, where we have only two involutive antimorphisms $R$ and $E$ and $G$ is generated by both of them. On one hand, the most famous $G$-rich sequence is the Thue-Morse sequence ${\bf t}$, which is however not rich in the classical sense \citep{PStarosta2013}. On the other hand, every complementary symmetric Rote sequence~\footnote{A binary sequence is a complementary symmetric Rote sequence if its language is closed under
letter exchange and it has factor complexity $C(n) = 2n$ for all $n\geq 1$~\citep{Rote1994}.} is $G$-rich and simultaneously rich in the classical sense.  

Let us emphasize that the glory of the Thue-Morse sequence rests in the fact that this sequence has the smallest critical exponent among all binary sequences. Recall that $r \in \mathbb{R}$ is the {\em critical exponent} of a~sequence ${\bf w}$ if no repetition $r'>r$ occurs in $\uu$ and $r$ is the largest number with this property.

Vesti in \citep{Vesti2014} suggested to study the critical exponent of sequences rich in palindromes. Baranwal and Shallit~\citep{BaSh19} focused on rich sequences over binary alphabet. They defined a~rich sequence $\uu$ with the critical exponent $2+\frac{\sqrt{2}}{2}$ and conjectured that no other binary rich sequence has a smaller critical exponent. This conjecture was proven by Currie, Mol and Rampersad~\citep{CuMoRa2020}. It is natural to call this value the {\em repetition threshold} of binary rich sequences. 
We explain in Section~\ref{sec:binary} that binary rich sequences reaching the repetition threshold are $G$-rich for the group $G$ generated by $R$ and $E$.

Recently,  Currie, Mol and Peltomäki~\citep{CuMoPe2024} constructed a sequence $\zz$ over the alphabet $\{0,1,2\}$ and showed that it has the smallest critical exponent among all ternary rich sequences. More precisely, they showed that the repetition threshold of ternary rich sequences equals $1+\frac{1}{3-\mu}\doteq 2.25876324$, where $\mu$ is the unique real root of the polynomial $x^3-2x^2-1$. The authors themselves noticed that the language of $\zz$ is closed under the antimorphism $S$ exchanging letters $1 \leftrightarrow 2$. We show that the sequence $\zz$ is $G$-rich for the group $G$ generated by $R$ and $S$, see Theorem~\ref{thm:G-richness of z} in Section~\ref{sec:ternary}. Their paper is very long, thus the task to find suitable candidates for rich sequences with the smallest critical exponent over larger alphabets will be certainly computationally expensive. Our result indicates that suitable candidates might be found among $G$-rich sequences with properly chosen group $G$.

\section{Preliminaries}
An \textit{alphabet} $\mathcal A$ is a finite set, its elements are \textit{letters}. A \textit{word} $u$ over $\mathcal A$ of \textit{length} $n$ is a finite sequence $u = u_0 u_1 \cdots u_{n-1}$ of letters $u_j\in\mathcal A$ for all $j \in \{0,1,\dots, n-1\}$. The length of $u$ is denoted $|u|$. The set of all finite words over $\A$ is denoted $\A^*$. The set $\A^*$ equipped with concatenation as the operation forms a monoid with the \textit{empty word} $\varepsilon$ as the neutral element. Consider $u, p, s, v \in \A^*$ such that $u=pvs$, then the word $p$ is called a \textit{prefix}, the word $s$ a \textit{suffix} and the word $v$ a \textit{factor} of $u$. 
A~\textit{sequence} $\uu$ over $\A$ is an infinite sequence $\uu = u_0 u_1 u_2 \cdots$ of letters $u_j \in \A$ for all $j \in \N$. A \textit{word} $w$ over $\mathcal A$ is called a~\textit{factor} of the sequence $\uu = u_0 u_1 u_2 \cdots$ if there exists $j \in \mathbb N$ such that $w = u_j u_{j+1} u_{j+2} \cdots u_{j+|w|-1}$. If $j=0$, then $w$ is a \textit{prefix} of $\uu$.

The \textit{language} $\mathcal{L}(\uu)$ of a sequence $\uu$ is the set of factors occurring in $\uu$, the set of factors of length $n$ is denoted $\mathcal{L}_n(\uu)$.
The \textit{factor complexity} of a sequence $\uu$ is a mapping ${\mathcal C}:\mathbb N \to \mathbb N$, where $${\mathcal C}(n)=\#{\mathcal L}_n(\uu)\,.$$

A~factor $w$ of a sequence $\uu$ is \textit{left special} if $iw, jw \in \mathcal{L}(\uu)$ for at least two distinct letters ${i, j} \in \A$. A \textit{right special} factor is defined analogously.

A sequence $\uu$ is \textit{recurrent} if each factor of $\uu$ has infinitely many occurrences in $\uu$. Moreover, a recurrent sequence $\uu$ is \textit{uniformly recurrent} if for every $n \in \mathbb N$ there exists $N\in \mathbb N$ such that each factor of $\uu$ of length $N$ contains all factors of $\uu$ of length $n$. 

A \textit{morphism} is a map $\psi: \A^* \to \B^*$ such that $\psi(uv) = \psi(u)\psi(v)$ for all words $u, v \in \A^*$.
The morphism $\psi$ can be naturally extended to a sequence $\uu=u_0 u_1 u_2\cdots$ over $\A$ by setting
$\psi(\uu) = \psi(u_0) \psi(u_1) \psi(u_2) \cdots\,$.
If a morphism $\psi: \A^* \to \A^*$ satisfies $\psi(i) \not =\varepsilon$ for every letter $i \in \mathcal A$ and there exists $j \in \A$ and $w\in\A^*, w\not =\varepsilon$, such that $\psi(j)=jw$, then there exists a sequence $\uu$ having the prefix $\psi^n(j)$ for every $n \in \mathbb N$, thus $\psi(\uu)=\uu$ and $\uu$ is a~\textit{fixed point} of $\psi$. 
In the sequel, we use the notation $\psi^{\omega}(j)$.
A~morphism $\psi: \A^* \to \A^*$ is \textit{primitive} if there exists $k\in \mathbb N$ such that $\psi^k(i)$ contains all letters of $\mathcal A$ for each letter $i \in \mathcal A$. It is known that a fixed point of a primitive morphism is uniformly recurrent.

An \textit{antimorphism} is a map $\Theta: \A^* \to \A^*$ such that $\Theta(uv) = \Theta(v)\Theta(u)$ for all words $u, v \in \A^*$.
The \textit{reversal} mapping $R:\A^* \to \A^*$ is an antimorphism satisfying $R(i)=i$ for each letter $i \in \mathcal A$.
The language $\mathcal{L}(\uu)$ is called \textit{closed under} $\Theta$ if for each factor $w$, the word $\Theta(w)$ is also a factor of $\uu$.
A word $w \in \A^*$ is a $\Theta$-\textit{palindrome} if $w=\Theta(w)$. If a uniformly recurrent sequence $\uu$ contains infinitely many $\Theta$-palindromic factors, then its language ${\mathcal L}(\uu)$ is closed under $\Theta$. 
The \textit{$\Theta$-palindromic complexity} of a sequence $\uu$ is a~mapping ${\mathcal P}_\Theta:\mathbb N \to \mathbb N$, where $${\mathcal P}_\Theta(n)=\#\{w \in {\mathcal L}_n(\uu) \ : \ w =\Theta(w)\}\,.$$ 

\section{The rich ternary sequence $\zz$ with the minimum critical exponent as a morphic image of a fixed point}
Currie, Mol and Peltomäki~\citep{CuMoPe2024} showed that the repetition threshold for ternary rich sequences is reached by the sequence $\zz$, defined below. Moreover, every sequence that reaches the threshold has the same language as $\zz$. In the sequel, we keep notation from~\citep{CuMoPe2024}.
Let the morphisms $f, g: \{0,1,2\}^* \mapsto \{0,1,2\}^*$ be defined by letter images
$$\begin{array}{ccc}
f: \left\{\begin{array}{lll}
    0&\to & 01  \\
    1&\to & 022  \\    
    2&\to & 02  
\end{array}\right. & \qquad \text{and}\qquad &
g: \left\{\begin{array}{lll}
    0&\to & 20  \\
    1&\to & 21  \\    
    2&\to & 2  
\end{array}\right.\,.
\end{array}$$
Denote $\xx=f^\omega(0)$ and  $\yy = g(\xx)= y_0y_1y_2 \cdots$. 

As the last step, a transducer $\tau$ is applied to $\yy$.
To describe the action of the transducer $\tau$ we denote 
\begin{equation}\label{eq:ABCD}
A= 00101101, \quad B = 001, \quad C = 00202202 \text{ and } D = 002. 
\end{equation} $\tau$  maps the letters standing at even positions and odd positions in a different way, namely
$$\tau(y_{2i}) = \left\{\begin{array}{ll}B& \text{if \ } y_{2i} = 0\\
A& \text{if \ } y_{2i} = 1\\
AA& \text{if \ } y_{2i} = 2\end{array}\right.\qquad \text{and}\qquad \tau(y_{2i+1}) = \left\{\begin{array}{ll}D& \text{if \ } y_{2i+1} = 0\\
C& \text{if \ } y_{2i+1} = 1\\
CC& \text{if \ } y_{2i+1} = 2\end{array}\right.\,.
$$  
The sequence \begin{equation}\label{eq:Z}\zz = \tau(\yy)= \tau\bigl(g(f^\omega(0))\bigr)\end{equation}is a rich sequence with the minimum critical exponent.

\bigskip

We will show that $\zz$ may be obtained as a morphic image of a fixed point. Consider two morphisms $\varphi: \{0,1,2,3,4\}^*\mapsto \{0,1,2,3,4\}^*$ and  $\psi: \{0,1,2,3,4\}^*\mapsto \{0,1,2\}^*$ 
$$\begin{array}{ccc}
\varphi: \left\{\begin{array}{lll}
    0&\to & 01  \\
    1&\to & 02  \\    
    2&\to & 03\\ 
 3&\to & 04\\ 
4&\to & 044
\end{array}\right. & \qquad \text{and}\qquad &
\psi: \left\{\begin{array}{lll}
    0&\to & 0  \\
    1&\to & 1  \\    
    2&\to & 22\\ 
 3&\to & 202\\ 
4&\to & 20102
\end{array}\right.\,.
\end{array}$$
Denote $\uu=\varphi^{\omega}(0)$, a prefix of $\uu$ reads\\
\noindent $\uu= 0102010301020104010201030102010440102010301020104010201030102010440440102\cdots
$

To reach our goal, i.e., to show that $\zz$ is a morphic image of a fixed point, the following lemma is handy.

\begin{lemma}\label{lem:1} $f^n\psi = \psi \varphi^n$ for every  $n \in \N$. 
    \end{lemma}
    \begin{proof} It is easy to see that $f \psi = \psi \varphi$ as both morphisms map  
    
    \centerline{$ 0\mapsto 01, \ \ 1 \mapsto 022,  \ \ 2 \mapsto 0202, \ \ 3 \mapsto 020102, \ \ 4 \mapsto 02010220102 $\,.}

      Let us complete the proof by induction on $n \in \N$. We have just shown the statement for $n=1$ and it is trivial for $n=0$. 
        Let $n\geq 2$. Using the induction hypothesis and the validity of the statement for $n=1$, we get 

$$f^n \psi =f \bigl(f^{n-1}\psi \bigr) =f \bigl(\psi \varphi^{n-1} \bigr)= \bigl(f\psi\bigr)\varphi^{n-1} =\bigl(\psi \varphi\bigr) \varphi^{n-1} = \psi \varphi^n\,. 
$$
  \end{proof}
\begin{proposition}  Using words $A,B,C,D$ from~\eqref{eq:ABCD}, we can write  $\zz = \xi( \varphi^\omega(0))$, where $\xi:\{0,1,2,3,4\}^*\to \{0,1,2\}^*$ is a morphism $$\xi:\ \ \left\{\begin{array}{ccl}
     0& \mapsto& A\\
    1& \mapsto &AD \\
     2& \mapsto &AC  \\
      3& \mapsto & ACC \\
       4& \mapsto & ACCBCC\\
\end{array}\right..$$     
In particular, $\zz$ is uniformly recurrent. 
\end{proposition}

\begin{proof} By Lemma \ref{lem:1}, $g\bigl(f^\omega(0)\bigr) = g \bigl(f^\omega(\psi(0))\bigr) = g\bigl(\psi \varphi^\omega(0)\bigr)= (g\psi)\varphi^\omega(0)$.  The composition $g\psi$  maps letters of $\{0,1,2,3,4\}$ as follows
\medskip

    \centerline{$ 0\mapsto 20, \ \ 1 \mapsto 21,  \ \ 2 \mapsto 22, \ \ 3 \mapsto 2202, \ \ 4 \mapsto 22021202 $.}
\medskip
    \noindent 
As $\yy =g\bigl(f^\omega(0)\bigr)= (g\psi)\varphi^\omega(0)$ and the length of 
 $g\psi(a)$ is even for every letter $a \in \{0,1,2,3,4\} $,  we may give an explicit formula for $\tau g\psi$
$$\tau g\psi:\ \ \left\{\begin{array}{ccl}
     0& \mapsto& AA D\\
    1& \mapsto &AAC  \\
     2& \mapsto &AACC  \\
      3& \mapsto & AACCBCC \\
       4& \mapsto & AACCBCCACCBCC\\
\end{array}\right.\,.$$ 
It is straightforward  to check that $\tau g \psi = \xi \varphi$.  Hence $\zz = \tau(\yy) = (\tau g \psi) \varphi^\omega(0) = (\xi\varphi)(\varphi^{\omega}(0))=\xi \bigl(\varphi (\varphi^\omega(0)\bigr) = \xi \bigl(\varphi^\omega(0)\bigr)$.  
Uniform recurrence follows from the fact that fixed points of primitive morphisms are uniformly recurrent.   
\end{proof}

\section{Definition of generalized richness of sequences} 
In this section, we mainly cite definitions and results from~\citep{PStarosta2014}.
Throughout the paper, let $G$ be a finite group of morphisms and antimorphisms on $\mathcal{A}^*$ such that $G$ contains at least one antimorphism.  We say that a sequence $\uu$ over $\mathcal{A}$ is closed under $G$ if, for every $\eta \in G$ and every factor $w$ of $\uu$, the word $\eta(w)$ is a factor of $\uu$, too. On the set of factors of $\uu$ we define an equivalence as usual $$ w\sim v \quad \Longleftrightarrow \quad w=\eta(v)\ \text{for some} \ \eta \in G.
 $$
Then $[w]$ denotes the class of equivalence containing $w$, i.e., $[w]$ is the orbit of $w$ under the action of the group $G$. 
 
To recall the definition of $G$-richness, we need to introduce the graph of symmetries of $\uu$.    

\begin{defi}\label{de:grafSymetrie} Let $\uu$ be a sequence with language closed under $G$ and let $n \in N$.  The undirected graph of symmetries of the sequence $\uu$ of order $n$ is ${\Gamma}_n(\uu) = (V, E)$ with the set of vertices 
    $$ V =\{[w]: w\in \mathcal{L}_n(\uu), \ \ w \ \text{\ is left or right special}  \}
    . $$ Two vertices $[w]\in V$ and  $[v]\in V$  are connected by  an edge $[e] \in E $ if 
    there exists a factor $u \in [e] \subset \mathcal{L}(\uu)$ of length $>n$ such that
    \begin{enumerate}
        \item the prefix of $u$ of length $n$ belongs to $[w]$,
        \item the suffix of $u$ of length $n$ belongs to $[v]$,
        \item $u$ contains no other special factor of length $n$ (except the prefix and the suffix of length n). 
    \end{enumerate} 
\end{defi}
Note that the role of $[w]$  and $[v]$ in the previous definition of edges is symmetric, as the language of $\uu$ is closed under $G$.

\begin{example}  Let us construct $\Gamma_2(\uu)=(V, E)$  for the group $G= \{I,R\}$, where $\uu$ is the fixed point of $\varphi$. 
The left and right special factors of $\uu$ of length 2 are: $01,10,04,40$, thus $V=\{[01], [04]\}$. 

\noindent Factors of length $>2$ satisfying items 1., 2. and 3. in Definition \ref{de:grafSymetrie} are: 

\medskip

\centerline{$010, \ 10201, \ 10301, \ 040, \ 0440, \ 104,  \ 401$.}

\medskip
\noindent Hence the vertex $[01]$ has three loops, namely $[010], [10201], [10301]$, the vertex $[04]$ has two loops, namely $[0440]$ and $[040]$. The vertices $[04]$ and $[01]$ are connected by one edge $[104]$, see Figure~\ref{fig:symmetry_graph}. 
    \end{example}

\begin{figure}[h!]
\begin{center}
\begin{tikzpicture}[every loop/.style={},node distance=3cm,
                    thick,main node/.style={circle,draw,font=\sffamily\small\bfseries}]

  \node[main node] (1) {$[01]$};
  \node[main node] (2) [right of=1] {$[04]$};

  \path[every node/.style={font=\sffamily\small}]
    (1) edge node [above] {$[104]$} (2)
     (2) edge [loop right] node [right] {$[040]$} (2)
     (2) edge [loop above] node [above] {$[0440]$} (2)
    (1) edge [loop left] node [left] {$[010]$} (1)
       edge [loop above] node [above]{$[10201]$} (1)
       edge [loop below] node [below]{$[10301]$} (1);

\end{tikzpicture}
\caption{The graph of symmetries ${\Gamma}_2(\uu)$ for $\uu=\varphi^{\omega}(0)$ and the group $G=\{I,R\}$.}\label{fig:symmetry_graph}
\end{center}
\end{figure}
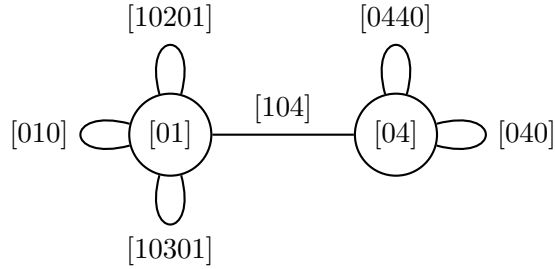
\begin{defi}\label{def:tls} Let $\uu$ be a sequence with the language closed under $G$.
 We say that $\uu$ has Property $G$-$tls(N)$, if for every $n \geq  N$ the following two conditions are satisfied: 
 
 \begin{itemize}
     \item the graph $\Gamma_\uu(n)$ after removing the loops is a tree; 
     \item if an edge $[e]$ is a loop in $\Gamma_\uu(n)$, then $e$ is a $\Theta$-palindrome for some antimorphism $\Theta \in G$. 
 \end{itemize}
 
\end{defi}
Note that $tls$ is an abbreviation for ``tree like structure''. 

\begin{defi}\label{def:Grichness} Let $\uu$ be a sequence with the language closed under $G$.
 We say that $\uu$ is $G$-rich,  if $\uu$ has Property $G$-$tls(1)$.  
\end{defi}

Recall that a recurrent sequence is rich if and only if it is $G$-rich for $G=\{I,R\}$.

Sequences rich in the classical sense may be characterized by various ways: using complete return words~\citep{GlJuWi2009},  using occurrences of the longest palindromic suffixes of words,   the extensions of bispecial factors~\citep{BaPeSt2010}, etc. All these characterizations also have analogies for $G$-richness, see \citep{PStarosta2014}. 
One of the tools for proving $G$-richness is based on the relation between palindromic and factor complexity. For the classical richness this relation was observed and proved in~\citep{BuLuGlZa2009}. 
 
The set of involutive antimorphisms of $G$ is denoted  $G^{(2)}$, i.e.,  
$$G^{(2)} = \{\psi\in G: \psi \text{ is an antimorphism and }\psi^2 =I   \}\,.$$ 
\begin{remark} In general, $G$ may contain also antimorphisms that are not involutive, however, in this paper, we will deal only with groups whose  antimorphisms are all involutive.  
\begin{itemize}
\item For ``classical'' richness, where $G$  contains only identity $I$ and reversal mapping $R$, we have $G^{(2)} = \{R\}$. 

\item In Section~\ref{sec:ternary}, we focus on the group $G=\{I, R, S, RS\}$, where  $R$ and $S$ are antimorphisms on the ternary alphabet $\{0,1,2\}$. $R$ denotes the reversal mapping and $S$ is defined  by $S(0)=0, S(1)=2$  and $S(2)=1$.  In this case $G^{(2)} = \{R, S\}$. 

\item In Section~\ref{sec:binary}, the considered group $G=\{I, R, E, ER\}$, where $R$ is the reversal mapping and $E$ is the antimorphism exchanging letters. In this case $G^{(2)} = \{R, E\}$.
\end{itemize}

\end{remark} 
 
We say that $N\in \N$ is $G$-distinguishing on $\uu$ if, for every factor $w \in \mathcal{L}_N(\uu)$ and every pair of antimorphisms $\Theta_1,\Theta_2 \in G$, the following implication holds: $ \Theta_1(w) = \Theta_2(w) \Rightarrow \Theta_1 = \Theta_2$. 

\begin{example}\label{ex:disting3} Consider  the group $G= \{I, R, S, RS\}$.  Any factor of length $3$ occurring in $\zz$ contains either the letter $1$ or $2$. Thus $R(w)\neq S(w)$ if $|w|\geq 3$ and therefore $N=3$ is $G$-distinguishing on $\zz$. 
    
\end{example}

\begin{theorem}[\citep{PStarosta2013}]\label{thm:Nerovnost}
 Let $\uu$ be a sequence with language closed under $G$ and let $N\in \N$ be $G$-distinguishing on $\uu$. 
 \begin{enumerate}
     \item For every $n \geq N$,  \begin{equation}\label{eq:nepreleze}\mathcal{C}(n+1)- \mathcal{C}(n) +\#G \geq \sum_{\Theta \in G^{(2)}} \mathcal{P}_{\Theta}(n) +\mathcal{P}_{\Theta}(n+1). \end{equation}
     
     \item $\uu$ has Property $G$-$tls(N)$ if and only if the equality is attained for all $n\geq N$ in~\eqref{eq:nepreleze}. 
\end{enumerate}

\end{theorem}

\section{G-richness of the rich ternary sequence $\zz$}\label{sec:ternary}

In the paper~\citep{CuMoPe2024} it is proven that $\zz$ is rich in the classical sense. 
Our goal is to show that $\zz$ is also $G$-rich, where $G= \{I,R,S,RS\}$. Let us summarize the properties of $\zz$ deduced in~\citep{CuMoPe2024}. 
\begin{proposition}\label{pro:richnessZ} Let $\zz$ be the sequence   defined in \eqref{eq:Z}. Denote by $ \mathcal{C}$ and $\mathcal{P}_R$  its  factor and  $R$-palindromic complexity function, respectively.   Then 
\begin{enumerate}
    \item $\mathcal{C}(0)=1$, $\mathcal{C}(1)=3$, $\mathcal{C}(2)=7$,  $\mathcal{C}(3)=12$ and 
$$\mathcal{C}(n) = 4n+2\quad \text{ for all } \  n\geq 4.$$
\item   $\mathcal{P}_R(0) = 1$, $\mathcal{P}_R(1) = \mathcal{P}_R(2) = 3$,  $\mathcal{P}_R(3) = 4$ and 
$$\mathcal{P}_R(n) =\left\{ \begin{array}{cl}
   2  & \text{ if $n\geq 5$ and $n$ is odd;}  \\
   4  & \text{ if $n\geq 4$ and $n$ is even.}
\end{array}\right. $$
\item The language of $\zz$ is closed under the antimorphism $S$. 
\end{enumerate} 

\end{proposition}

To show $G$-richness of $\zz$, we have to examine $S$-palindromic complexity of $\zz$. The following lemma will be helpful for this purpose.
\begin{lemma}\label{lem: constructionS_palindromes}
Let $w$ be a factor of $\varphi^{\omega}(0)$. 
If $\xi(w)00$ is an $S$-palindrome, then $\xi(\varphi(w))00$ is an $S$-palindrome, too. 
\end{lemma}
\begin{proof}
Let us prove the statement by induction on the length of $|w|$. 
If $w=\varepsilon$, the statement is trivial.
If $|w|=1$, then the only $S$-palindrome of the discussed form is $\xi(2)00=AC00$, where $A,B,C,D$ are given in~\eqref{eq:ABCD}, and the reader may easily check that $\xi(\varphi(2))00=\xi(03)00=AACC00$ is an $S$-palindrome, too. If $|w|=2$, then the only $S$-palindrome of the discussed form is $\xi(03)00$ and the reader may again check that $\xi(\varphi(03))00=\xi(0104)00=AADAACCBCC00$ is an $S$-palindrome, too. 
Now, consider $w$ of length $|w|\geq 3$. By the form of $\xi$ and $\varphi$ and since the only factors of length two of $\varphi^{\omega}(0)$ are $01, 02, 03, 04, 44$ and their mirror images, we observe that if $\xi(w)00$ is an $S$-palindrome, then there exists a~non-empty factor $w'$ of $\varphi^{\omega}(0)$ such that $w$ is of one of the following forms:
$$w\in \{0104, 0102w'4, 0103w'04, 0104w'0104, 02w'3, 03w'03, 04w'0103, 2w'2, 3w'02, 4w'0102\}\,.$$
Let us comment on the most complicated case. If $w$ starts in $01$, then $\xi(w)$ starts in $AAD$. Since $\xi(w)00$ is an $S$-palindrome, $\xi(w)$ ends in $BCC$, but then $w$ necessarily ends in $4$. This implies that $\xi(w)00$ starts in $AADAAC$ and ends in $ACCBCC$. We have several possibilities:
\begin{itemize}
\item $\xi(w)00=AADAACCBCC$, which leads to $w=0104$;
\item $\xi(w)00=AADAACACCBCC$, which means that $w=01024$, but $01024$ is not a~factor of $\varphi^{\omega}(0)$;
\item $\xi(w)00=AADAACA\cdots CACCBCC$, which leads to the form $w=0102w'4$ for some non-empty factor $w'$ of $\varphi^{\omega}(0)$;
\item $\xi(w)00=AADAACCAACCBCC$, which means that $w=010304$, but $010304$ is not a factor of $\varphi^{\omega}(0)$;
\item $\xi(w)00=AADAACCA\cdots CAACCBCC$, which means that $w=0103w'04$ for some non-empty factor $w'$ of $\varphi^{\omega}(0)$;
\item $\xi(w)00=AADAACCBCCAADAACCBCC$, which gives $w=01040104$, but $01040104$ is not a factor of $\varphi^{\omega}(0)$;
\item $\xi(w)00=AADAACCBCC\cdots AADAACCBCC$, which means that $w=0104w'0104$ for some non-empty factor $w'$ of $\varphi^{\omega}(0)$.
\end{itemize}

In any case, it is readily seen that $\xi(w')00$ is an $S$-palindrome, too. 
Since $|w'|<|w|$, by induction assumption, $\xi(\varphi(w'))00$ is an $S$-palindrome. It follows then by mechanical verification that $\xi(\varphi(w))00$ is an $S$-palindrome, too. 


\end{proof}

\begin{proposition}\label{prop:Spalindromes}  Let $\zz$ be a sequence defined in~\eqref{eq:Z}. Denote by $\mathcal{P}_S$  its  $S$-palindromic complexity function. Then ${P}_S(2n) \geq 2$ for every $n\geq 2$. 
\end{proposition}
\begin{proof}
According to Lemma~\ref{lem: constructionS_palindromes}, the word $x_k=\xi(\varphi^k(2))00$ for $k\in \mathbb N$ is an $S$-palindrome. It suffices to observe that $\xi(2)00=AC00$ is an $S$-palindrome, where $A,C$ are defined in~\eqref{eq:ABCD}. Moreover, $x_k$ is obviously a factor of $\zz$ and $x_k$ is of even length. An $S$-palindrome of odd length could contain only $0$ as its central factor, however none of the words $000$, $102$, $201$ is a factor of $\zz$. Consequently, for any $n\geq 2$, the central factor of length $2n$ of the word $x_k$ and its reversal are distinct $S$-palindromes (they contain both letters $1$ and $2$, which guarantees their difference).
\end{proof}
\begin{theorem}\label{thm:G-richness of z} The sequence $\zz$ is $G$-rich.     
\end{theorem}

\begin{proof}
By Definition~\ref{def:Grichness} we have to show that for every $n\in \N, n\geq 1,$ the graph $\Gamma_n(\zz)$ after removing loops is a~tree and that every loop in $\Gamma_n(\zz)$ is an $R$-palindrome or an $S$-palindrome. 

We start by showing that it is true for $n=1$ and $n=2$ and then we complete the proof by showing that $\zz$ has Property $G$-$tls(3)$. 
Consider $n=1$.  All letters are left or right special factors, hence the graph has two vertices $[0]$ and $[1]=\{1,2\}$, there is one edge $[01]=\{01,10,02,20\}$ between those two vertices, the only loop corresponding to $0$ is $[00]$ and the only loop corresponding to $[1]$ is $[11]=\{11,22\}$. Both loops are formed by $R$-palindromes. See Figure~\ref{fig:SmallGraphs} (left).

Consider $n=2$. There are two vertices of $\Gamma_2(\zz)$: $[00]$ and $[01]$ and the only edge connecting them is $[001]$.
The loops corresponding to $[01]$ are $[010]=\{010, 020\}, \ [101]=\{101, 202\}$ and $[0110]=\{0110, 0220\}$ and they all are $R$-palindromes. See Figure~\ref{fig:SmallGraphs} (right).

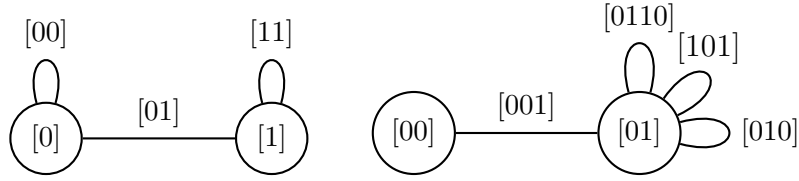
\begin{figure}[h!]
\begin{center}
\begin{tabular}{cccc}
\begin{tikzpicture}[every loop/.style={},node distance=3cm,
                    thick,main node/.style={circle,draw,font=\sffamily\small\bfseries}]

  \node[main node] (1) {$[0]$};
  \node[main node] (2) [right of=1] {$[1]$};

  \path[every node/.style={font=\sffamily\small}]
    (1) edge node [above] {$[01]$} (2)
     (1) edge [loop above] node [above] {$[00]$} (1)
     (2) edge [loop above] node [above] {$[11]$} (2);

\end{tikzpicture} && \begin{tikzpicture}[every loop/.style={},node distance=3cm,
                    thick,main node/.style={circle,draw,font=\sffamily\small\bfseries}]

  \node[main node] (1) {$[00]$};
  \node[main node] (2) [right of=1] {$[01]$};

  \path[every node/.style={font=\sffamily\small}]
    (1) edge node [above] {$[001]$} (2)
     (2) edge [loop right] node [right] {$[010]$} (2)
     (2) edge [loop above] node [above] {$[0110]$} (2);
   
    \path (2) edge [out=25,in=55,looseness=8] node [above] {$[101]$} (2);

\end{tikzpicture}
\end{tabular}
\caption{The graphs of symmetries ${\Gamma}_1(\zz)$ (left) and ${\Gamma}_2(\zz)$ (right) for $\zz$ defined in~\eqref{eq:Z} and the group $G=\{I,R, S, RS\}$.}\label{fig:SmallGraphs}
\end{center}
\end{figure}

According to Example~\ref{ex:disting3}, the number $N=3$ is $G$-distinguishing on $\zz$. Using the second part of Theorem~\ref{thm:Nerovnost}, to deduce Property $G$-$tls(3)$ it is necessary and sufficient to prove for every $n \geq 3$ the equality 
\begin{equation}\label{eq:CoTreba}
\mathcal{C}(n+1)- \mathcal{C}(n) +4 = \mathcal{P}_{R}(n) +\mathcal{P}_{R}(n+1) + \mathcal{P}_{S}(n) +\mathcal{P}_{S}(n+1). 
\end{equation}
We will make use of Propositions~\ref{pro:richnessZ} and~\ref{prop:Spalindromes}.

Inserting $n=3$, we have on the left hand side $\mathcal{C}(4)- \mathcal{C}(3) +4 =  6+4 = 10$ and on the right hand side $\mathcal{P}_{R}(3) +\mathcal{P}_{R}(4) + \mathcal{P}_{S}(3) +\mathcal{P}_{S}(4)\geq 4 + 4 +2 = 10$. By the first part of Theorem~\ref{thm:Nerovnost}, the equality~\eqref{eq:CoTreba} follows. 

If $n\geq 4$, then we have on the left hand side $\mathcal{C}(n+1)- \mathcal{C}(n) +4 =  4+4 = 8$ and on the right hand side
$\mathcal{P}_{R}(n) +\mathcal{P}_{R}(n+1) + \mathcal{P}_{S}(n) +\mathcal{P}_{S}(n+1)\geq 4 + 2+ 2 = 8$. Again, the first part of Theorem~\ref{thm:Nerovnost} forces the equality \eqref{eq:CoTreba}. 

\end{proof}

\section{G-richness of the rich binary sequences with the minimum critical exponent}\label{sec:binary}
The authors of~\citep{CuMoRa2020} showed that, up to a letter permutation, there are two sequences with the minimum critical exponent over the binary alphabet $\{0,1\}$, namely $f(h^{\omega}(0))$ and $f(g(h^{\omega}(0)))$, where the morphisms $g,h:\{0,1,2\}^*\to \{0,1,2\}^*$ and $f:\{0,1,2\}^*\to \{0,1\}^*$ are defined by letter images as follows
$$\begin{array}{ccc}
f: \left\{\begin{array}{lll}
    0&\to & 0  \\
    1&\to & 01  \\    
    2&\to & 011
\end{array}\right. & 
g: \left\{\begin{array}{lll}
    0&\to & 011  \\
    1&\to & 0121  \\    
    2&\to & 012121
\end{array}\right. & 
h: \left\{\begin{array}{lll}
    0&\to & 01  \\
    1&\to & 02  \\    
    2&\to & 022
\end{array}\right.\,.
\end{array}$$
They also proved that both $f(h^{\omega}(0))$ and $f(g(h^{\omega}(0)))$ are complementary symmetric Rote sequences. 
Pelantová and Starosta~\citep{PeSt2016} proved that every complementary symmetric Rote sequence is $G$-rich for the group $G=\{I, R, E, ER\}$, where $E$ is the antimorphism defined by $E(0)=1$ and $E(1)=0$. Hence, also in the binary case, the rich sequences with the minimum critical exponent are $G$-rich.

\section{Open problem}
A natural question arises. Is the repetition threshold of rich sequences on alphabets of size larger than three also reached by $G$-rich sequences for some groups $G$ containing more than one antimorphism?


\end{document}